\documentclass[12pt,oneside,american]{amsart}
\usepackage{tgschola}

\usepackage[T1]{fontenc}
\usepackage[latin9]{inputenc}
\usepackage{geometry}
\usepackage{color}
\usepackage{babel}
\usepackage{prettyref}
\usepackage{units}
\usepackage{mathrsfs}
\usepackage{mathtools}
\usepackage{amstext}
\usepackage{amsthm}
\usepackage{amssymb}
\usepackage{stackrel}
\usepackage{setspace}
\usepackage[unicode=true,pdfusetitle,
 bookmarks=true,bookmarksnumbered=false,bookmarksopen=false,
 breaklinks=false,pdfborder={0 0 1},backref=false,colorlinks=true]
 {hyperref}
\usepackage{breakurl}

\makeatletter

\newcommand{\noun}[1]{\textsc{#1}}

\numberwithin{equation}{section}
\numberwithin{figure}{section}
\numberwithin{table}{section}
\theoremstyle{plain}
\newtheorem{thm}{\protect\theoremname}[section]
\theoremstyle{plain}
\newtheorem*{thm*}{\protect\theoremname}
\theoremstyle{definition}
\newtheorem*{defn*}{\protect\definitionname}
\theoremstyle{remark}
\newtheorem{rem}[thm]{\protect\remarkname}
\theoremstyle{plain}
\newtheorem{prop}[thm]{\protect\propositionname}
\theoremstyle{plain}
\newtheorem*{lem*}{\protect\lemmaname}
\theoremstyle{plain}
\newtheorem{cor}[thm]{\protect\corollaryname}

\usepackage{thmtools}

\makeatother

\providecommand{\corollaryname}{Corollary}
\providecommand{\definitionname}{Definition}
\providecommand{\lemmaname}{Lemma}
\providecommand{\propositionname}{Proposition}
\providecommand{\remarkname}{Remark}
\providecommand{\theoremname}{Theorem}

\begin{document}

\title{On the Dirichlet-to-Neumann operator for the Connection Laplacian}

\author{Ravil Gabdurakhmanov}

\address{School of Mathematics, University of Leeds, Leeds LS2 9JT, UK}

\email{mmrg@leeds.ac.uk}

\begin{abstract}
We study the relationship between the symbol of the Dirichlet-to-Neumann
operator associated with a connection Laplacian, and the geometry
on and near the boundary. As a consequence, we show that the geometric
data on the boundary, and when the dimension of the base is greater
than two all corresponding normal derivatives, are determined by the
symbol.
\end{abstract}

\maketitle

\section{Introduction}

Let $\left(N,g\right)$ be a compact connected Riemannian manifold
with non-emtpy boundary $\partial N$, and let $E\rightarrow N$ be
a Euclidean vector bundle endowed with a compatible connection $\nabla^{E}$.
Consider the connection Laplacian $\triangle^{E}$ associated with
the connection $\nabla^{E}$. It is a natural generalization of the
Laplace-Beltrami operator. We define the corresponding Dirichlet-to-Neumann
(DtN) operator $\Lambda_{g,\nabla^{E}}$ by sending a section $\sigma$
on the boundary $\partial N$ to the outward normal covariant derivative
of its harmonic extension.

In this paper we study the DtN operator $\Lambda_{g,\nabla^{E}}$
as a pseudodifferential operator on the boundary. We follow the strategy
in the paper \cite{key-6} by Lee and Uhlmann for the Laplace-Beltrami
operator. They showed that the Riemannian metric on the boundary can
be recovered from a given DtN operator. In addition, if the dimension
of the manifold is greater than $2$, they showed that all the normal
derivatives of the metric can be recovered as well. This result was
used in \cite{key-6} to recover a Riemannian manifold from the DtN
operator under some assumptions on the geometry and topology of a
manifold, and subsequently in \cite{key-7,key-5} under assumptions
on the geometry only. This method also appears in the work of Ceki\'{c}
\cite{key-1} on Calderon's problem for Yang-Mills connections.

It is well known that a Riemannian metric on a manifold can be recovered
from the Laplacian. This can be done by considering the principal
symbol of the Laplacian which is equal to $\left|\xi\right|_{g\left(x\right)}^{2}$,
where $\xi\in T_{x}^{*}N$, $x\in N$. The Dirichlet-to-Neumann operator
is a classical elliptic pseudodifferential operator of order one on
the boundary. Therefore, it is natural to use the same idea for the
recovery of the geometric data on the boundary from the DtN operator.
There is a local factorization of the Laplacian into the composition
of two operators near the boundary which establishes the relationship
between the symbols of the DtN operator and Laplacian. In particular,
it turns out that the principal symbol of the DtN operator is the
(minus) square root of the principal symbol of the boundary Laplacian.
Therefore, the principal symbol of the DtN operator is equal to $-\left|\xi\right|_{\left.g\right|_{\partial N}}$,
and it is straightforward to determine the metric from it. The rest
part of the symbol is expressed in terms of the local geometric data
in a more sophisticated way. By analyzing these expressions we are
able to recover the geometric data on the boundary from the full symbol
of the DtN operator. More precisely, we prove the following result.
\begin{thm}
\label{thm:1}Suppose $dim\,N=n\ge3$. Let $\left(x^{1},\dots,x^{n-1}\right)$
be any local coordinates for an open set $W\subset$$\partial N$
and $\left(\epsilon_{1},\dots,\epsilon_{r}\right)$ any local frame
of $E$ over $W$, and let $\left\{ \lambda_{j},\,j\le1\right\} $
be the full symbol of the DtN operator $\Lambda$ in these coordinates
and local frame. For any $p\in W$, the full Taylor series of $g$
and $\nabla^{E}$ at $p$ in boundary normal coordinates and boundary
normal frame is given by an explicit formula in terms of the matrix
functions $\left\{ \lambda_{j}\right\} $ and their tangential derivatives
at $p$.
\end{thm}

On surfaces, the DtN operator naturally scales under the conformal
changes of a metric. As a consequence, we cannot recover the normal
derivatives of a metric and a connection at the boundary. So the result
in this case is a bit weaker.
\begin{thm}
\label{thm:2}Let $\left(N,g\right)$ be a Riemannian surface. Let
$\left(x^{1}\right)$ be any local coordinate for an open set $W\subset$$\partial N$
and $\left(\epsilon_{1},\dots,\epsilon_{r}\right)$ any local frame
of $E$ over $W$. Let $\left\{ \lambda_{j},\,j\le1\right\} $ be
the full symbol of the DtN operator $\Lambda$ in these coordinate
and local frame. Then for any $p\in W$ the metric $g$ and the connection
$\nabla^{E}$ at $p$ in these coordinate and frame is given by an
explicit formula in terms of the matrix functions $\left\{ \lambda_{j}\right\} $
and their tangential derivatives at $p$.
\end{thm}

The structure of this paper is as follows. In Section \ref{sec:Background-material}
we recall some background material including the definition and properties
of the connection Laplacian and associated DtN operator. In Section
\ref{sec:Reconstruction-of-the} we obtain some provisional results
such as a local representation of the connection Laplacian in boundary
normal coordinates and boundary normal frame, the local factorization
of the connection Laplacian in these coordinates, and the relation
of this factorization to the DtN operator. We then use this results
to prove Theorem \ref{thm:1}. The relationship between the geometry
of isomorphic vector bundles, where the isomorphism intertwines the
DtN operators, and the discussion of the two-dimensional case conclude
the section.

\section{\label{sec:Background-material}Background material}

\subsection{Notation}

Let $\left(N,g\right)$ be a compact connected Riemannian manifold
with boundary $\partial N$ and $\left(E,\nabla^{E}\right)$ be a
vector bundle over $N$ of rank $r$ with a connection $\nabla^{E}:\Gamma\left(E\right)\rightarrow\Gamma\left(E\otimes T^{*}N\right)$.
We assume that $E$ is equipped with a compatible inner product $\left\langle \cdot,\cdot\right\rangle _{E}$,
the relation 
\[
d\left\langle u,v\right\rangle _{E}=\left\langle \nabla^{E}u,v\right\rangle _{E}+\left\langle u,\nabla^{E}v\right\rangle _{E}
\]
holds for any pair of smooth sections $u,v\in\Gamma\left(E\right)$.
Note that both sides of the above identity are differential forms.
We use the standard notation $\Gamma\left(E\right)$ for $C^{\infty}$-smooth
sections of a vector bundle $E$, and $H^{s}\left(E\right)$ for $H^{s}$-smooth
sections, where $H^{s}=W^{s,2}$ and $W^{s,p}$ denotes the $\left(s,p\right)$-Sobolev
space. Note that in this notation $H^{0}\left(E\right)=L^{2}\left(E\right)$.
The space of test sections of $E$ is denoted as usually by $\mathcal{D}\left(E\right)$.
The continuous dual of this space is denoted by $\mathcal{D}'\left(E\right)$. 

We can define the $L^{2}-$inner product of sections by
\[
\left\langle u,v\right\rangle _{L^{2}}=\underset{N}{\int}\left\langle u,v\right\rangle _{E}dV_{g},
\]
where $dV_{g}$ is the Riemannian volume measure of $\left(N,g\right)$.
Similarly, we can define the $L^{2}-$inner product in $\Gamma\left(E\otimes T^{*}N\right)$
by 
\[
\left\langle \alpha,\beta\right\rangle _{L^{2}\left(E\otimes T^{*}N\right)}=\underset{N}{\int}Tr_{g}\left\langle \alpha,\beta\right\rangle _{E}dV_{g},
\]
where the sections $\alpha,\beta\in\Gamma\left(E\otimes T^{*}N\right)$
considered as the $E-$valued $1-$forms.

Let us consider the connection Laplacian $\triangle^{E}$ defined
by
\[
\triangle^{E}=-Tr_{g}\,\bar{\nabla}^{E}\nabla^{E},
\]
where $\bar{\nabla}^{E}=\nabla^{E}\otimes\nabla^{LC}:\Gamma\left(E\otimes T^{*}N\right)\rightarrow\Gamma\left(E\otimes T^{*}N\otimes T^{*}N\right)$
and $\nabla^{LC}$ is the Levi-Civita connection on $\left(N,g\right)$.
Note that we have the following equality \cite{key-9}

\[
\triangle^{E}=\left(\nabla^{E}\right)^{\ast}\nabla^{E},
\]
where $\left(\nabla^{E}\right)^{\ast}$ is the adjoint of $\nabla^{E}$
with respect to the $L_{2}-$inner products defined above. We will
occasionally omit the word ``connection'' and call this operator
the Laplacian for brevity. When it will not make any confusion, we
will also sometimes omit the superscript $E$ in the notation of the
connection $\nabla^{E}$.

Let us consider the Dirichlet problem for the Laplacian on a Euclidean
vector bundle $E$ over a compact connected Riemannian manifold $\left(N,g\right)$
with boundary $\partial N$:
\begin{equation}
\begin{cases}
\triangle^{E}u=0 & on\,\,N,\\
\left.u\right|_{\partial N}=\sigma & u\in\Gamma\left(E\right),\,\sigma\in\Gamma\left(\left.E\right|_{\partial N}\right).
\end{cases}\label{eq:3.1}
\end{equation}
It has a unique solution for every $\sigma\in\Gamma\left(\left.E\right|_{\partial N}\right)$.
We briefly explain why this is true in the next section. Now this
allows us to introduce the \emph{Dirichlet-to-Neumann operator} $\Lambda_{g,\nabla^{E}}:\Gamma\left(\left.E\right|_{\partial N}\right)\rightarrow\Gamma\left(\left.E\right|_{\partial N}\right)$
associated with the connection Laplacian $\nabla^{E}$ by 
\[
\Lambda_{g,\nabla^{E}}\sigma=\left.\nabla_{\nu}^{E}u\right|_{\partial N},
\]
where $u$ is the solution to \ref{eq:3.1}, and $\nu$ is the outward
unit normal vector field on $\partial N$.

\subsection{Discussion on the Dirichlet problem}

We start with well-known Green's identities. Let $w\in\Gamma\left(E\otimes T^{*}N\right)$
and $v\in\Gamma\left(E\right)$. Then we have the first Green's identity
\begin{gather*}
\left\langle \nabla^{*}w,v\right\rangle _{L^{2}}-\left\langle w,\nabla v\right\rangle _{L^{2}\left(E\otimes T^{*}N\right)}=\underset{\partial N}{\int}\left\langle \iota_{\nu}w,v\right\rangle dS_{g},
\end{gather*}
where $dS_{g}$ is the associated volume form on the boundary $\partial N$,
and $\nu$ is the outward unit normal vector field on $\partial N$.
Now let $w=\nabla u$, where $u\in\Gamma\left(E\right)$. Then we
obtain the following identity for the Laplacian
\begin{equation}
\left\langle \triangle^{E}u,v\right\rangle _{L^{2}}=\left\langle \nabla^{*}\nabla u,v\right\rangle _{L^{2}}=\left\langle \nabla u,\nabla v\right\rangle _{L^{2}\left(E\otimes T^{*}N\right)}+\underset{\partial N}{\int}\left\langle \iota_{\nu}\nabla u,v\right\rangle dS_{g},\label{eq:2.1(1)}
\end{equation}
which gives us the second Green's identity\emph{
\begin{equation}
\left\langle \triangle^{E}u,v\right\rangle _{L^{2}}-\left\langle u,\triangle^{E}v\right\rangle _{L^{2}}=\underset{\partial N}{\int}\left\langle \iota_{\nu}\nabla u,v\right\rangle dS_{g}-\underset{\partial N}{\int}\left\langle u,\iota_{\nu}\nabla v\right\rangle dS_{g}.\label{eq:2.1(2)}
\end{equation}
}

From \eqref{eq:2.1(1)} we conclude that the kernel of the Dirichlet
Laplacian is zero, since
\[
0=\left\langle \triangle^{E}u,u\right\rangle _{L^{2}}=\left\langle \nabla u,\nabla u\right\rangle _{L^{2}\left(E\otimes T^{*}N\right)}+\underset{\partial N}{\int}\left\langle \iota_{\nu}\nabla u,u\right\rangle dS_{g}=\left\langle \nabla u,\nabla u\right\rangle _{L^{2}\left(E\otimes T^{*}N\right)}=\left\Vert \nabla u\right\Vert _{L^{2}\left(E\otimes T^{*}N\right)},
\]
which implies
\[
\nabla u=0,
\]
and therefore
\[
d\left\langle u,u\right\rangle _{E}=2\left\langle \nabla u,u\right\rangle _{E}=0.
\]
From this we conclude that $\left\langle u,u\right\rangle _{E}$ is
constant on $N$ and since $u$ vanishes on the boundary it has to
be identically zero on the whole manifold $N$. It is clear that this
implies the uniqueness of the solution to the boundary value problem
\eqref{eq:3.1}.

The existence of the solution can be shown in two steps. First, we
use the Lax-Milgram theorem to prove the existence of a weak solution
to the Dirichlet problem
\begin{equation}
\begin{cases}
\triangle^{E}u=-\triangle^{E}\tilde{\sigma}, & u\in H^{1}\left(E\right),\\
\left.u\right|_{\partial N}=0,
\end{cases}\label{eq:2.2-1}
\end{equation}
where $\tilde{\sigma}\in H^{1}\left(E\right)$ in the right hand side
is a given function. Then we use the elliptic regularity to show that
this solution is smooth if $\tilde{\sigma}$ is. Below we briefly
discuss both steps.

Let us recall the Lax-Milgram theorem \cite{key-4}.
\begin{thm*}[Lax-Milgram]
 Let $V$ be a Hilbert space and $a\left(\cdot,\cdot\right)$ a bilinear
form on $V$, which is
\begin{enumerate}
\item Bounded: $\left|a\left(u,v\right)\right|\leq C\left\Vert u\right\Vert \left\Vert v\right\Vert $
and
\item Coercive: $\left|a\left(u,u\right)\right|\geq c\left\Vert u\right\Vert ^{2}$.
\end{enumerate}
Then for any continuous linear functional $f\in V^{\prime}$ there
is a unique solution $u\in V$ to the equation
\[
a\left(u,v\right)=f\left(v\right),\,\forall v\in V
\]
and the following inequality holds
\[
\left\Vert u\right\Vert \leq\frac{1}{c}\left\Vert f\right\Vert _{V^{\prime}}.
\]
\end{thm*}
Let $\mathring{H}^{1}\left(E\right)\coloneqq\mathring{H}^{1,2}\left(E\right)$
denotes the Sobolev space of sections vanishing on the boundary $\partial N$.
Consider a bilinear form $a\left(\cdot,\cdot\right)$ which acts on
$u,v\in\mathring{H}^{1}\left(E\right)$ as

\[
a\left(u,v\right)=\left\langle \nabla u,\nabla v\right\rangle _{L^{2}\left(E\otimes T^{*}N\right)}.
\]
We see that $a\left(u,v\right)$ is a bounded bilinear form on $\mathring{H}^{1}\left(E\right)$.
For a given $\tilde{\sigma}\in H^{1}\left(E\right)$ let us take the
functional 
\[
f_{\tilde{\sigma}}\left(v\right)=-\left\langle \nabla\tilde{\sigma},\nabla v\right\rangle _{L^{2}\left(E\otimes T^{*}N\right)}.
\]
 Clearly, this is a bounded linear functional on $\mathring{H}^{1}\left(E\right)$.
So it is left to show that $a\left(u,v\right)$ is coercive, i.e.
there exists $c$ such that
\[
\frac{\left\Vert \nabla u\right\Vert _{L^{2}\left(E\otimes T^{*}N\right)}^{2}}{\left\Vert u\right\Vert _{L^{2}}^{2}}\geq c>0.
\]
This follows, for instance, from the Sobolev embedding theorem \cite{key-3}.
We showed that all the conditions of the Lax-Milgram theorem are satisfied,
therefore, there is a unique weak solution $u\in\mathring{H}^{1}\left(E\right)$
to \eqref{eq:2.2-1}. Let us now state the elliptic regularity theorem
(see, for example, \cite{key-3}).
\begin{thm*}[Elliptic regularity]
 Suppose $\varphi\in H^{s}\left(E\right)$ (smooth), and $u\in\mathscr{D}^{\prime}\left(E\right)$
solves the equation
\[
\triangle^{E}u=\varphi.
\]
Then $u\in H^{s+2}\left(E\right)$ (respectively smooth).
\end{thm*}
Using this theorem we conclude that if $\tilde{\sigma}$ is smooth
then the solution $u$ to \eqref{eq:2.2-1} is also smooth.

To obtain the solution to the boundary value problem \eqref{eq:3.1}
we take a smooth extension $\tilde{\sigma}\in\Gamma\left(E\right)$
of $\sigma$, and write $u=w+\tilde{\sigma}$. We see that 
\[
\begin{cases}
\triangle^{E}u=\triangle^{E}w+\triangle^{E}\tilde{\sigma},\\
\left.u\right|_{\partial N}=\left.w\right|_{\partial N}+\left.\tilde{\sigma}\right|_{\partial N}=\left.w\right|_{\partial N}+\sigma.
\end{cases}
\]
Now if $w$ is the solution to \eqref{eq:2.2-1}, then $u$ is the
solution to \eqref{eq:3.1}. Note that the restriction to the boundary
extends to the trace map 
\begin{equation}
\mathcal{T}:H^{s}\left(E\right)\rightarrow H^{s-\frac{1}{2}}\left(\left.E\right|_{\partial N}\right),\label{eq:TrMp}
\end{equation}
which has the right inverse 
\begin{equation}
\mathcal{E}:H^{s-\frac{1}{2}}\left(\left.E\right|_{\partial N}\right)\rightarrow H^{s}\left(E\right),\label{eq:InTrMp}
\end{equation}
i.e. the map such that the equality$\mathcal{T\circ E}=Id$ holds;
both maps are linear bounded operators. Using this we can show the
existence and uniqueness of the solution to the boundary value problem
\begin{equation}
\begin{cases}
\triangle^{E}u=0, & u\in H^{s+\frac{1}{2}}\left(E\right),\\
\mathcal{T}\left(u\right)=\sigma, & \sigma\in H^{s}\left(\left.E\right|_{\partial N}\right),
\end{cases}\label{eq:2.5}
\end{equation}
for $s\geq\nicefrac{1}{2}$. Note that from the elliptic regularity
the solution $u$ is smooth in the interior of $N$.

Let us now consider the complexification $\mathbb{C}\otimes E\simeq E\oplus iE$
of the Euclidean vector bundle $E$. It has a natural structure of
Hermitian vector bundle and a compatible connection $\nabla=\nabla^{E}\oplus\nabla^{E}$.
The associated Laplacian is 
\[
\triangle=\triangle^{E}\oplus\triangle^{E},
\]
and we see that the complex analogue of problem \eqref{eq:2.5} decomposes
into two real ones. So all the results on the existence and uniqueness
are straightforward from the real case. We will further assume the
complexified objects (vector bundle, connection, Laplacian) and use
the same notation for them as for their real counterparts.

\subsection{Pseudodifferential operators on vector bundles.}

In this section we will recall the definition of a (standard) pseudodifferential
operator (PDO) on vector bundles. In this section we follow the exposition
by Treves in \cite{key-2}. Let $\Omega$ be an open subset of $\mathbb{R}^{n}$.
We need first to define the special class of functions called \emph{amplitudes}.
Let $m$ be any real number. We shall denote by $S^{m}\left(\Omega,\Omega\right)$
the linear space of $C^{\infty}$ functions in $\Omega\times\Omega\times\mathbb{R}^{n}$,
$a\left(x,y,\xi\right)$, which have the following property:

To every compact subset $\mathscr{K}$ of $\Omega\times\Omega$ and
to every triplet of $n$-tuples $\alpha,\beta,\gamma$, there is a
constant $C_{\alpha,\beta,\gamma}\left(\mathscr{K}\right)>0$ such
that
\begin{equation}
\left|D_{\xi}^{\alpha}D_{x}^{\beta}D_{y}^{\gamma}a\left(x,y,\xi\right)\right|\leq C_{\alpha,\beta,\gamma}\left(\mathscr{K}\right)\left(1+\left|\xi\right|\right)^{m-\left|\alpha\right|},\,\forall\left(x,y\right)\in\mathscr{K},\textrm{and }\forall\xi\in\mathbb{R}^{n},\label{eq:1.5}
\end{equation}
where $D_{z_{i}}\coloneqq-i\partial_{z_{i}}$. The elements of $S^{m}\left(\Omega,\Omega\right)$
are called amplitudes of degree $\leq m$ (in $\Omega\times\Omega$).
The space $S^{m}\left(\Omega,\Omega\right)$ is endowed with a natural
locally convex topology: denote by $\mathfrak{p}_{\mathscr{K};\alpha,\beta,\gamma}\left(a\right)$
the infimum of the constants $C_{\alpha,\beta,\gamma}\left(\mathscr{K}\right)$
such that \eqref{eq:1.5} is true. One can see then that the function
$\mathfrak{p}_{\mathscr{K};\alpha,\beta,\gamma}$ is a seminorm on
$S^{m}\left(\Omega,\Omega\right)$ and defines the topology of this
space when $\mathscr{K}$ ranges over the collection of all compact
subsets of $\Omega$ and $\alpha,\beta,\gamma$ over that of all $n$-tuples.
Thus topologized, $S^{m}\left(\Omega,\Omega\right)$ is a Fr�chet
space. We denote the subspace of amplitudes independent of $y$ by
$S^{m}\left(\Omega\right)$ and regard its elements as smooth functions
in $\Omega\times\mathbb{R}^{n}$ (rather than in $\Omega\times\Omega\times\mathbb{R}^{n}$).
The topology on $S^{m}\left(\Omega\right)$ is the induced subspace
topology from $S^{m}\left(\Omega,\Omega\right)$. Hence, $S^{m}\left(\Omega\right)$
is a Fr�chet space, as a closed linear subspace of the Fr�chet space.
The intersection of all $S^{m}\left(\Omega\right)$ is denoted by
$S^{-\infty}\left(\Omega\right)$. The quotient vector space $S^{m}\left(\Omega\right)/S^{-\infty}\left(\Omega\right)$
is denoted by $\dot{S}^{m}\left(\Omega\right)$ and its elements are
called \emph{symbols }of degree $\leq m$. We use the term symbol
for a representative ($\in S^{m}\left(\Omega\right)$) of an equivalence
class in $\dot{S}^{m}\left(\Omega\right)$ as well. We will also need
the notion of a formal symbol. By a \emph{formal symbol }we mean a
sequence of symbols $a_{m_{j}}\in S^{m_{j}}\left(\Omega\right)$ whose
orders $m_{j}$ are strictly decreasing and converging to $-\infty$.
It is standard to represent it by the formal series 
\begin{equation}
\stackrel[j=0]{+\infty}{\sum}a_{m_{j}}\left(x,\xi\right).\label{eq:1.6}
\end{equation}
From such a formal symbol one can build true symbols, elements of
$S^{m_{0}}\left(\Omega\right)$ in the present case, which all belong
to the same class modulo $S^{-\infty}\left(\Omega\right)$. We will
denote this class by \eqref{eq:1.6}. In order to construct a true
symbol one may proceed as follows:

First, one may state that a symbol $a\left(x,\xi\right)$ belongs
to the class \eqref{eq:1.6} if, given any large positive number $M$,
there is an integer $J\geq0$ such that
\[
a\left(x,\xi\right)-\stackrel[j=0]{J}{\sum}a_{m_{j}}\left(x,\xi\right)\in S^{-M}\left(\Omega\right).
\]
Second, one can construct such a symbol $a\left(x,\xi\right)$ as
a sum of a series
\[
\stackrel[j=0]{+\infty}{\sum}\chi_{j}\left(\xi\right)a_{m_{j}}\left(x,\xi\right),
\]
where $\chi_{j}\left(\xi\right)$ are suitable cutoff functions. By
suitable we mean such that the above series converges in the space
$S^{m_{0}}\left(\Omega\right)$. Note that since we are using the
cutoff functions we are able to deal with terms $a_{m_{j}}\left(x,\xi\right)$
that are not ``true'' elements of $S^{m_{j}}\left(\Omega\right)$,
e.g. the functions that are non-smooth or even not defined in neighborhoods
of the origin $\xi=0$. (If such neighborhoods depend on $x$ we may
consider the cutoff functions that also depend on $x$). The most
important examples of formal symbols \eqref{eq:1.6} with terms $a_{m_{j}}\left(x,\xi\right)$
that are not $C^{\infty}$ functions of $\xi$ at the origin are the
\emph{classical symbols.} The formal symbol \eqref{eq:1.6} is called
the classical symbol if each term $a_{m_{j}}\left(x,\xi\right)$ is
a positive-homogeneous function of degree $m_{j}$ of $\xi$ and if
differences $m_{j}-m_{j+1}\in\mathbb{Z}^{+}$, for all $j$. Recall
that a function $f\left(\xi\right)$ is called \emph{positive-homogeneous}
of degree $d$ with respect to $\xi$ in $\mathbb{R}^{n}\backslash\left\{ 0\right\} $
if $f\left(\rho\xi\right)=\rho^{d}f\left(\xi\right)$ for every $\rho>0$
(but not necessarily for every real $\rho$). For instance, the Heaviside
function on $\mathbb{R}\backslash\left\{ 0\right\} $ is positive-homogeneous,
but not homogeneous, of degree zero.

Let us continue with the definition of a (standard) PDO. A linear
operator $A:\mathscr{E}^{\prime}\left(\Omega\right)\rightarrow\mathscr{D}^{\prime}\left(\Omega\right)$
is called a (standard) pseudodifferential operator of order $m$ if
there is an amplitude $a\in S^{m}\left(\Omega,\Omega\right)$ such
that $A$ can be represented in the form
\[
Au\left(x\right)=\left(2\pi\right)^{-n}\int\int e^{i\left(x-y\right)\cdot\xi}a\left(x,y,\xi\right)u\left(y\right)dyd\xi.
\]

Now we are able to define pseudodifferential operators acting on vector-valued
distributions. Let $\mathbb{F}^{n}$ be a $n$-dimensional vector
space over field $\mathbb{F}$($=\mathbb{R},\mathbb{C}$). Let $\mathscr{D}^{\prime}\left(\Omega;\mathbb{F}^{n}\right)$
($\mathscr{E}^{\prime}\left(\Omega;\mathbb{F}^{n}\right)$) denote
the space of (compactly supported) distributions in $\Omega$. Note
that if we chose basis in $\mathbb{F}^{n}$ then a $\mathbb{F}^{n}$-valued
distribution is a vector with coordinates consisting of $n$ scalar
distributions. Using this we can give the following natural definition.
A linear operator $A:\mathscr{E}^{\prime}\left(\Omega;\mathbb{F}^{r}\right)\rightarrow\mathscr{D}^{\prime}\left(\Omega;\mathbb{F}^{l}\right)$
is called a pseudodifferential operator if in any bases of $\mathbb{F}^{r}$
and $\mathbb{F}^{l}$ it is represented by a matrix of scalar pseudodifferential
operators. 

Let $F_{1}$ and $F_{2}$ be two vector bundles over a smooth manifold
$N$ with fibers $\mathbb{F}^{r}$ and $\mathbb{F}^{l}$, respectively.
Note that any $F_{i}$-valued distribution ($i=1,2$) in any local
trivialization is represented by a vector with coordinates being scalar
distributions. This allows to represent any linear operator $\Gamma\left(F_{1}\right)\rightarrow\Gamma\left(F_{2}\right)$
in any local trivializations as an $l\times r$-matrix of linear operators
$C^{\infty}\left(N\right)\rightarrow C^{\infty}\left(N\right)$. Therefore,
it is natural to give the following definition. 
\begin{defn*}
A linear operator $A:\mathscr{E}^{\prime}\left(F_{1}\right)\rightarrow\mathscr{D}^{\prime}\left(F_{2}\right)$
is called a pseudodifferential operator (of order $m$) if in any
pair of local trivializations it is represented by a matrix of scalar
pseudodifferential operators (of order $m$). 
\end{defn*}
If vector bundles $F_{1}$ and $F_{2}$ are equal to the vector bundle
$F$, then we say that $A$ is a PDO (acting) on a vector bundle $F$.
Due to this definition most of the theory for scalar PDOs generalizes
to this case naturally, e.g. the symbol calculus. In particular, if
$A$ and $B$ are two PDOs on vector bundle $F$, then the symbol
of the composition $A\circ B$ is defined by the formal symbol
\begin{equation}
\underset{\alpha}{\sum}\frac{1}{\alpha!}\partial_{\xi}^{\alpha}a\left(x,\xi\right)D_{x}^{\alpha}b\left(x,\xi\right),\label{eq:SymbOfComp}
\end{equation}
where $a\left(x,\xi\right)$ and $b\left(x,\xi\right)$ are the symbols
of $A$ and $B$, respectively.

We shall say that a PDO $A$ is smoothing if it maps $\mathscr{E}^{\prime}\left(\Omega\right)$
to $C^{\infty}\left(\Omega\right)$. In order for this to be the case,
it is necessary and sufficient for the associated kernel $K_{A}\left(x,y\right)$
to be $C^{\infty}$ in $\Omega\times\Omega$. 

\subsection{\label{subsec:Well-posedness-of-the}Well-posedness of the generalized
heat equation and regularity of its solution.}

In this subsection we describe the result by Treves \cite[III.1]{key-2}.
In order to be precise we introduce the original setting. Let $X$
be a smooth manifold; $n=\dim X$; $t$ be the variable in the closed
half-line $\overline{\mathbb{R}}_{+}$; $T$ be some positive real
number.

We shall deal with functions and distributions valued in a finite-dimensional
Hilbert space $H$ over $\mathbb{C}$. The norm in $H$ will be denoted
by $\left|\cdot\right|_{H}$, whereas the operator norm in $L\left(H\right)$,
the space of (bounded) linear operators in $H$, will be denoted by
$\left\Vert \cdot\right\Vert $. The inner product in $H$ will be
denoted by $\left(\cdot,\cdot\right)_{H}$. The space of $H$-valued
distributions in $X$ will be denoted by $\mathscr{D}^{\prime}\left(X;H\right)$.

Let $A\left(t\right)$ be a pseudodifferential operator of order $1$
in $X$, valued in $L\left(H\right)$, depending smoothly on $t\in\left[0,T\right)$.
If we fix basis in $H$, then $A\left(t\right)$ is a matrix whose
entries are scalar pseudodifferential operators in $X$. This means
that in every local chart $\left(\Omega,x_{1},\dots,x_{n}\right)$,
$A\left(t\right)$ is congruent modulo smoothing operators which are
$C^{\infty}$-functions of $t$ to an operator
\[
A_{\Omega}\left(t\right)u\left(x\right)=\left(2\pi\right)^{-n}\int e^{ix\cdot\xi}a_{\Omega}\left(x,t,\xi\right)\hat{u}\left(\xi\right)d\xi,\quad u\in C_{c}^{\infty}\left(\Omega;H\right),
\]
where $a_{\Omega}\left(x,t,\xi\right)$ is a smooth function of $t\in\left[0,T\right)$
valued in $S^{1}\left(\Omega;L\left(H\right)\right)$, the space of
symbols valued in $L\left(H\right)$.

According to Treves \cite[III.1]{key-2}, the heat equation for $A\left(t\right)$
is well-posed and possesses a regularity property described below
if the following conditions are satisfied:
\begin{enumerate}
\item Let $\left(\Omega,x_{1},\dots,x_{n}\right)$ be any local chart in
$X$. There is a symbol $a_{\Omega}\left(x,t,\xi\right)$ satisfying
\[
a_{\Omega}\left(x,t,\xi\right)\textrm{ is a \ensuremath{C^{\infty}} function of \ensuremath{t\in\left[0,T\right)} valued in \ensuremath{S^{1}\left(\Omega;L\left(H\right)\right)}},
\]
and defining the operator $A_{\Omega}\left(t\right)$ congruent to
$A\left(t\right)$ modulo smoothing operators in $\Omega$ depending
smoothly on $t\in\left[0,T\right)$, such that
\item to every compact subset $K$ of $\Omega\times\left[0,T\right)$ there
is a compact subset $K^{\prime}$ of the open half-plane $\mathbb{C}_{-}=\left\{ z\in\mathbb{C};\,\textrm{Re}z<0\right\} $
such that
\item the map
\[
z\cdot Id-\frac{a_{\Omega}\left(x,t,\xi\right)}{\left(1+\left|\xi\right|^{2}\right)^{\nicefrac{1}{2}}}:H\rightarrow H
\]
is a bijection (hence also a homeomorphism), for all $\left(x,t\right)\in K$,
$\xi\in\mathbb{R}^{n}$, $z\in\mathbb{C}\setminus K^{\prime}$.
\end{enumerate}
The regularity property that we are interested in is described in
the following theorem \cite[III, Theorem 1.2]{key-2}.
\begin{thm}
\label{thm:Regularity=00005BTreves=00005D}Let $\mathscr{O}$ be an
open subset of $X$, $u$ a $C^{\infty}$ function of $t$ in $\left[0,T\right)$
valued in $\mathscr{D}^{\prime}\left(X;H\right).$

Suppose that $u\left(0\right)\in C^{\infty}\left(\mathscr{O};H\right)$
and that
\[
\frac{\partial u}{\partial t}-A\left(t\right)u\in C^{\infty}\left(\mathscr{O}\times\left[0,T\right);H\right).
\]
Then $u\in C^{\infty}\left(\mathscr{O}\times\left[0,T\right);H\right)$.
\end{thm}

This theorem is one of the main ingredients in the proof of Proposition
\ref{prop:A is DtN}, which relates the symbol of the DtN operator
to the symbol of the connection Laplacian.

\section{\label{sec:Reconstruction-of-the}Reconstruction of the geometric
data on the boundary}

In this section we find the relation between full symbols of the DtN
operator and the connection Laplacian, and then use this relation
to prove Theorem \ref{thm:1}. We follow the general strategy used
in \cite{key-6} for the DtN operator associated with the Laplace-Beltrami
operator.

\subsection{Local factorization of the connection Laplacian}

Let us recall the construction of geodesic coordinates with respect
to the boundary. For each $q\in\partial N$, let $\gamma_{q}:\left[0,\epsilon\right)\rightarrow N$
denote the unit-speed geodesic starting at $q$ and normal to $\partial N$.
If $\left\{ x^{1},\dots,x^{n-1}\right\} $ are any local coordinates
for $\partial N$ near $p\in\partial N$, we can extend them smoothly
to functions on a neighborhood of $p$ in $N$ by letting them be
constant along each normal geodesic $\gamma_{q}$. If we then define
$x^{n}$ to be the parameter along each $\gamma_{q}$, it follows
that $\left\{ x^{1},\dots,x^{n}\right\} $ form coordinates for $N$
in some neighborhood of $p$, which we call the \emph{boundary normal
coordinates }determined by $\left\{ x^{1},\dots,x^{n-1}\right\} $.
In these coordinates $x^{n}>0$ in the interior of $N$, and $\partial N$
is locally characterized by $x^{n}=0$. The metric in these coordinates
has the form
\[
g=\stackrel[i,j=1]{n-1}{\sum}g_{ij}\left(x^{1},\dots,x^{n}\right)dx^{i}dx^{j}+\left(dx^{n}\right)^{2}.
\]
Let $\left(\epsilon_{1},\dots,\epsilon_{r}\right)$ be a smooth local
frame of $\left.E\right|_{\partial N}$ near $p\in\partial N$, we
can extend it to a smooth local frame $\left(\varepsilon_{1},\dots,\varepsilon_{r}\right)$
in a neighborhood of $p$ in $N$ by means of parallel transport along
each $\gamma_{q}$, i.e. for each $q$ we find the unique solution
to the parallel transport equation 
\begin{align*}
\nabla_{\dot{\gamma}_{q}}^{E}\varepsilon_{\alpha} & =0,\\
\left.\varepsilon_{\alpha}\right|_{\gamma_{q}\left(0\right)} & =\epsilon_{\alpha},\,for\,\alpha=1,\dots,r.
\end{align*}
We call this frame the \emph{boundary normal frame} determined by
$\left(\epsilon_{1},\dots,\epsilon_{r}\right)$. In boundary normal
coordinates we have then 
\begin{equation}
\nabla_{\nicefrac{\partial}{\partial x^{n}}}^{E}\varepsilon_{\alpha}=0.\label{eq:2.2}
\end{equation}
In local frame a section $u$ is represented as a vector-valued function
on $N$ and the connection $\nabla^{E}$ acts as 
\[
\nabla_{\cdot}^{E}u=du\left(\cdot\right)+\omega\left(\cdot\right)u,
\]
where $\omega=\omega_{k}dx^{k}$ denotes the matrix of connection
forms of $\nabla^{E}$. From \eqref{eq:2.2} we have then
\[
\omega_{n}=\omega\left(\frac{\partial}{\partial x^{n}}\right)=0.
\]

\begin{rem}
If the connection $\nabla^{E}$ is compatible with an inner product
$\left\langle \cdot,\cdot\right\rangle _{E}$ on $E$, we have the
following relation
\begin{equation}
d\left\langle u,v\right\rangle _{E}=\left\langle \nabla^{E}u,v\right\rangle _{E}+\left\langle u,\nabla^{E}v\right\rangle _{E}.\label{eq:3.2-1}
\end{equation}
Note that if the frame $\left(\epsilon_{1},\dots,\epsilon_{r}\right)$
is orthonormal then the associated boundary normal frame is also orthonormal,
since we have
\begin{gather*}
\frac{\partial}{\partial t}\left\langle \varepsilon_{\alpha},\varepsilon_{\beta}\right\rangle _{E}=\left\langle \nabla_{\dot{\gamma}}^{E}\varepsilon_{\alpha},\varepsilon_{\beta}\right\rangle _{E}+\left\langle \varepsilon_{\alpha},\nabla_{\dot{\gamma}}^{E}\varepsilon_{\beta}\right\rangle _{E}=0\\
\left.\left\langle \varepsilon_{\alpha},\varepsilon_{\beta}\right\rangle _{E}\right|_{t=0}=\left\langle \epsilon_{\alpha},\epsilon_{\beta}\right\rangle _{E}=0,\,\alpha\neq\beta,\\
\left.\left\langle \varepsilon_{\alpha},\varepsilon_{\alpha}\right\rangle _{E}\right|_{t=0}=\left\langle \epsilon_{\alpha},\epsilon_{\alpha}\right\rangle _{E}=1.
\end{gather*}
And applying \eqref{eq:3.2-1} to an orthonormal frame we obtain
\begin{multline*}
0=\frac{\partial}{\partial x^{i}}\left\langle \varepsilon_{\alpha},\varepsilon_{\beta}\right\rangle _{E}=\left\langle \varepsilon_{\tau}\omega_{\alpha}^{\tau}\left(\frac{\partial}{\partial x^{i}}\right),\varepsilon_{\beta}\right\rangle _{E}+\left\langle \varepsilon_{\alpha},\varepsilon_{\tau}\omega_{\beta}^{\tau}\left(\frac{\partial}{\partial x^{i}}\right)\right\rangle _{E}=\\
\left\langle \varepsilon_{\tau},\varepsilon_{\beta}\right\rangle _{E}\left(\omega_{\alpha}^{\tau}\right)_{i}+\left\langle \varepsilon_{\alpha},\varepsilon_{\tau}\right\rangle _{E}\left(\omega_{\beta}^{\tau}\right)_{i}=\left(\omega_{\alpha}^{\beta}\right)_{i}+\left(\omega_{\beta}^{\alpha}\right)_{i},\,for\,i=1,\dots,n,
\end{multline*}
and we see that $\left(\omega_{\beta}^{\alpha}\right)_{i}=-\left(\omega_{\alpha}^{\beta}\right)_{i}$,
i.e. the connection form is skew-symmetric.
\end{rem}

We will use further the following notation, $x=\left(x^{\prime},x^{n}\right)$,
$x^{\prime}=\left(x^{1},\dots,x^{n-1}\right)$, $\partial_{x^{j}}=\nicefrac{\partial}{\partial x^{j}}$,
$D_{x^{j}}=-i\partial_{x^{j}}$, and $D_{x}=\left(D_{x^{1}},\dots,D_{x^{n}}\right)$,
with similar definitions for $D_{x^{\prime}}$, $\partial_{x}$, and
$\partial_{x^{\prime}}$. The Einstein summation convention will be
assumed throughout this section.

In boundary normal coordinates and boundary normal frame, the connection
Laplacian is 
\[
\triangle^{E}u=\triangle u+g^{ij}\left[2\omega_{i}\partial_{x^{j}}u+\left(\left(\nabla_{\partial_{x^{i}}}\omega\right)\left(\partial_{x^{j}}\right)+\omega_{i}\omega_{j}\right)u\right],
\]
where $\left(g^{ij}\right)$ is the inverse of the matrix $\left(g_{ij}\right)$,
$\triangle$ is the (scalar) Laplace-Beltrami operator on $N$. We
can write
\[
Lu\coloneqq\triangle^{E}u=\left[\triangle+i\stackrel[j=1]{n-1}{\sum}V^{j}D_{x^{j}}+\tilde{Q}\right]u,
\]
where 
\begin{align*}
V^{j} & =2g^{jk}\omega_{k},\,j=1,\dots,n-1\\
\tilde{Q} & =\stackrel[i,j=1]{n-1}{\sum}g^{ij}\left[\left(\nabla_{\partial_{x^{i}}}\omega\right)\left(\partial_{x^{j}}\right)+\omega_{i}\omega_{j}\right].
\end{align*}

The Laplace-Beltrami operator in boundary normal coordinates can be
written as
\begin{multline*}
\triangle u=\stackrel[i,j=1]{n}{\sum}\varrho^{-\nicefrac{1}{2}}\partial_{x^{i}}\left(\varrho^{\nicefrac{1}{2}}g^{ij}u\right)=\partial_{x^{n}}\partial_{x^{n}}u+\frac{1}{2}\left(\partial_{x^{n}}\log\varrho\right)\partial_{x^{n}}u+\\
+\stackrel[i,j=1]{n-1}{\sum}\left(g^{ij}\partial_{x^{i}}\partial_{x^{j}}u+\frac{1}{2}g^{ij}\left(\partial_{x^{i}}\log\varrho\right)\partial_{x^{j}}u+\left(\partial_{x^{i}}g^{ij}\right)\partial_{x^{j}}u\right),
\end{multline*}
where $\varrho=\det\left(g_{ij}\right)$. Using this we can write
\begin{equation}
-L=-\triangle-iV^{l}D_{x^{l}}-\tilde{Q}=D_{x^{n}}^{2}+iF\left(x\right)D_{x^{n}}+Q\left(x,D_{x^{\prime}}\right),\label{eq:3.2}
\end{equation}
where
\begin{align*}
F\left(x\right) & =-\frac{1}{2}\stackrel[k,l=1]{n-1}{\sum}g^{kl}\left(x\right)\partial_{x^{n}}g_{kl}\left(x\right),\\
Q\left(x,D_{x^{\prime}}\right) & =\stackrel[k,l=1]{n-1}{\sum}g^{kl}\left(x\right)D_{x^{k}}D_{x^{l}}-i\stackrel[k,l=1]{n-1}{\sum}\left(\frac{1}{2}g^{kl}\left(x\right)\partial_{x^{n}}\log\varrho\left(x\right)+\partial_{x^{n}}g^{kl}\left(x\right)+V^{l}\right)D_{x^{l}}-\tilde{Q}.
\end{align*}

The Dirichlet-to-Neumann operator in boundary normal coordinates and
boundary normal frame is
\begin{equation}
\Lambda_{g,\nabla^{E}}\sigma=\left.\nabla_{\nu}^{E}u\right|_{\partial N}=\left.\left(\frac{\partial}{\partial x^{n}}+\omega_{n}\right)u\right|_{\partial N}=\left.\frac{\partial u}{\partial x^{n}}\right|_{\partial N}.\label{eq:DtN}
\end{equation}
The next Proposition shows that there is a useful local factorization
of the Laplacian into a composition of two first-order pseudodifferential
operators.
\begin{prop}
\label{prop:Factorization}There exists a pseudodifferential operator
$A\left(x,D_{x^{\prime}}\right)$ of order one in $x^{\prime}$ depending
smoothly on $x^{n}\in\left[0,T\right]$, for some $T>0$, such that
\begin{equation}
-L\equiv\left(D_{x^{n}}+iF\left(x\right)-iA\left(x,D_{x^{\prime}}\right)\right)\circ\left(D_{x^{n}}+iA\left(x,D_{x^{\prime}}\right)\right)\label{eq:3.6-1}
\end{equation}
modulo a smoothing operator. 
\end{prop}

\begin{proof}
We use the symbol calculus to construct such an operator $A\left(x,D_{x^{\prime}}\right)$.
From \eqref{eq:3.2} we get
\begin{multline}
L+\left(D_{x^{n}}+iF-iA\right)\circ\left(D_{x^{n}}+iA\right)=-D_{x^{n}}^{2}-iFD_{x^{n}}-Q+D_{x^{n}}^{2}+iFD_{x^{n}}-\\
-iAD_{x^{n}}+iD_{x^{n}}A-FA+AA=AA-Q+i\left[D_{x^{n}},A\right]-FA.\label{eq:3.3}
\end{multline}
Let $a$ denote the full symbol of $A\left(x,D_{x^{\prime}}\right)$
and $q$ denote the full symbol of $Q\left(x,D_{x^{\prime}}\right)$.
Then, by \eqref{eq:SymbOfComp} the full symbol of \eqref{eq:3.3}
is
\[
\underset{K}{\sum}\frac{1}{K!}\partial_{\xi}^{K}aD_{x}^{K}a-q+\partial_{x^{n}}a-Fa,
\]
and $q$ splits into three terms
\begin{multline*}
q\left(x,\xi^{\prime}\right)=\stackrel[k,l=1]{n-1}{\sum}g^{kl}\left(x\right)\xi_{k}\xi_{l}-\\
-i\stackrel[l=1]{n-1}{\sum}\left[\stackrel[k=1]{n-1}{\sum}\left(\frac{1}{2}g^{kl}\left(x\right)\partial_{x^{k}}\log\delta\left(x\right)+\partial_{x^{k}}g^{kl}\left(x\right)\right)+V^{l}\right]\xi_{l}-\tilde{Q}=\\
=q_{2}\left(x,\xi^{\prime}\right)+q_{1}\left(x,\xi^{\prime}\right)+q_{0}\left(x\right).
\end{multline*}
Let us write 
\[
a\left(x,\xi^{\prime}\right)\sim\underset{j\leq1}{\sum}a_{j}\left(x,\xi^{\prime}\right),
\]
where $a_{j}$ are positive-homogeneous of degree $j$ in $\xi^{\prime}$,
that is we will define $A$ by a formal symbol. We shall determine
$a_{j}$ recursively so that \eqref{eq:3.3} is zero modulo symbols
of smoothing operators.

The homogeneous terms of degree two in \eqref{eq:3.3} give us
\[
a_{1}a_{1}-q_{2}=0,
\]
so we can choose 
\begin{equation}
a_{1}=-\sqrt{q_{2}}.\label{eq:3.4}
\end{equation}
Note that $q_{2}$ and, therefore, also $a_{1}$ are scalar matrices.
The terms of degree one in \eqref{eq:3.3} give us
\[
a_{0}a_{1}+a_{1}a_{0}+\stackrel[l]{n-1}{\sum}\partial_{\xi^{l}}a_{1}D_{x^{l}}a_{1}-q_{1}+\partial_{x^{n}}a_{1}-Fa_{1}=0,
\]
and using relation \eqref{eq:3.4}, we get
\[
-2\sqrt{q_{2}}a_{0}+\stackrel[l]{n-1}{\sum}\partial_{\xi^{l}}\sqrt{q_{2}}D_{x^{l}}\sqrt{q_{2}}-q_{1}-\partial_{x^{n}}\sqrt{q_{2}}+F\sqrt{q_{2}}=0,
\]
thus we have
\begin{equation}
a_{0}=\frac{1}{2\sqrt{q_{2}}}\left[\stackrel[l]{n-1}{\sum}\partial_{\xi^{l}}\sqrt{q_{2}}D_{x^{l}}\sqrt{q_{2}}-q_{1}-\partial_{x^{n}}\sqrt{q_{2}}+F\sqrt{q_{2}}\right].\label{eq:3.5}
\end{equation}
The terms of degree zero in \eqref{eq:3.3} give us
\begin{multline*}
a_{-1}a_{1}+a_{1}a_{-1}+\stackrel[l]{n-1}{\sum}\partial_{\xi^{l}}a_{1}D_{x^{l}}a_{0}+\stackrel[l]{n-1}{\sum}\partial_{\xi^{l}}a_{0}D_{x^{l}}a_{1}+\\
+\frac{1}{2}\stackrel[k,l]{n-1}{\sum}\partial_{\xi^{k}}\partial_{\xi^{l}}a_{1}D_{x^{k}}D_{x^{l}}a_{1}-q_{0}+\partial_{x^{n}}a_{0}-Fa_{0}=0,
\end{multline*}
and using relation \eqref{eq:3.4} again, we get
\begin{multline}
a_{-1}=\frac{1}{2\sqrt{q_{2}}}\left[-\stackrel[l]{n-1}{\sum}\partial_{\xi^{l}}\sqrt{q_{2}}D_{x^{l}}a_{0}-\stackrel[l]{n-1}{\sum}\partial_{\xi^{l}}a_{0}D_{x^{l}}\sqrt{q_{2}}+\right.\\
\left.+\frac{1}{2}\stackrel[k,l]{n-1}{\sum}\partial_{\xi^{k}}\partial_{\xi^{l}}\sqrt{q_{2}}D_{x^{k}}D_{x^{l}}\sqrt{q_{2}}-q_{0}+\partial_{x^{n}}a_{0}-Fa_{0}\right],\label{eq:3.6}
\end{multline}
where $a_{0}$ is given by \eqref{eq:3.5}. Continuing the recursion
for the terms of degree $m\leq-1$ we have
\[
-2\sqrt{q_{2}}a_{m-1}+\underset{\substack{j,k,K\\
m\le j,k\le1\\
\left|K\right|=j+k-m
}
}{\sum}\frac{1}{K!}\partial_{\xi}^{K}a_{j}D_{x}^{K}a_{k}+\partial_{x^{n}}a_{m}-Fa_{m}=0.
\]
Therefore we get
\begin{equation}
a_{m-1}=\frac{1}{2\sqrt{q_{2}}}\left[\underset{\substack{j,k,K\\
m\le j,k\le1\\
\left|K\right|=j+k-m
}
}{\sum}\frac{1}{K!}\partial_{\xi}^{K}a_{j}D_{x}^{K}a_{k}+\partial_{x^{n}}a_{m}-Fa_{m}\right].\label{eq:3.7-1}
\end{equation}
\end{proof}
\begin{rem}
\label{rem:Extension of a manifold}Let $p\in\partial N$. Note that
we can extend $\left(N,g,E,\nabla^{E}\right)$ along the boundary
$\partial N$ near $p$. This means that there is a vector bundle
$\left(\tilde{E},\nabla^{\tilde{E}}\right)$ over a Riemannian manifold
$\left(\tilde{N},\tilde{g}\right)$ such that $N$ is included in
$\tilde{N}$ isometrically, the restriction of $\left(\tilde{E},\nabla^{\tilde{E}}\right)$
to $N$ coincides with $\left(E,\nabla^{E}\right)$, and the point
$p$ lies in the interior of $\tilde{N}$. Clearly, near $p\in\tilde{N}$
there is an extension of boundary normal coordinates (so that $x^{n}\in$$\left(-\epsilon,\epsilon\right)$)
and boundary normal frame. Due to the construction of $A\left(x,D_{x^{\prime}}\right)$
one sees that the factorization in Proposition \ref{prop:Factorization}
extends to a neighborhood of $p$ in $\tilde{N}$, i.e. there exists
a PDO $\tilde{A}\left(x,D_{x^{\prime}}\right)$ of order one in $x^{\prime}$
depending smoothly on $x^{n}\in\left[-\tilde{T},\tilde{T}\right]$,
for some positive $\tilde{T}<T$, such that it coincides with $A\left(x,D_{x^{\prime}}\right)$
for $x^{n}\in\left[0,\tilde{T}\right]$.
\end{rem}

\subsection{The full symbol of the Dirichlet-to-Neumann operator $\Lambda_{g,\nabla^{E}}$.}

Our next step is to relate the operator $A\left(x,D_{x^{\prime}}\right)$
with the Dirichlet-to-Neumann operator $\Lambda_{g,\nabla^{E}}$.
It turns out that this relation is quite simple.
\begin{prop}
\label{prop:A is DtN}The operator $A$ satisfies the following relation
\[
\left.A\left(x,D_{x^{\prime}}\right)\right|_{\partial N}\sigma\equiv\left.\partial_{x^{n}}u\right|_{\partial N}=\Lambda_{g,\nabla^{E}}\sigma
\]
 modulo a smoothing operator.
\end{prop}

\begin{proof}
Let $p\in\partial N$. Using Remark \ref{rem:Extension of a manifold}
we consider an extension $\left(\tilde{N},\tilde{g},\tilde{E},\nabla^{\tilde{E}}\right)$
along $\partial N$ near $p$. Choose a coordinate chart $\left(\Omega,x^{\prime}\right)$
in $\partial N$ containing $p$ and denote by $\left(x^{\prime},x^{n}\right)$
the corresponding boundary normal coordinates in $\tilde{N}$. Let
$\Omega^{\prime}\subset\Omega$ be a precompact open subset and $\sigma$
a section in $H^{\nicefrac{1}{2}}\left(\left.E\right|_{\partial N}\right)$
compactly supported in $\Omega^{\prime}$. Consider a solution $u\in H^{1}\left(\tilde{E}\right)\subset\mathscr{D}^{\prime}\left(\tilde{N};\tilde{E}\right)$
to
\[
\begin{cases}
\triangle^{\tilde{E}}u=0, & \text{\textrm{in }}\tilde{N}\\
\left.u\right|_{\partial N}=\sigma.
\end{cases}
\]
 By Proposition \ref{prop:Factorization} and Remark \ref{rem:Extension of a manifold},
this problem is locally equivalent to the following system of equalities
for $u,v$:
\begin{gather*}
\left(Id\cdot D_{x^{n}}+i\tilde{A}\left(x,D_{x^{\prime}}\right)\right)u=v,\quad\left.u\right|_{x^{n}=0}=\sigma,\\
\left(Id\cdot D_{x^{n}}+iF\left(x\right)-i\tilde{A}\left(x,D_{x^{\prime}}\right)\right)v=h\in C^{\infty}\left(\left[-\tilde{T},\tilde{T}\right]\times\Omega^{\prime};\mathbb{C}^{r}\right).
\end{gather*}
The second equation above can be viewed as a backwards generalized
heat equation; making the substitution $t=\tilde{T}-x^{n}$, it is
equivalent to
\begin{equation}
Id\cdot\partial_{t}v-\left(\tilde{A}-F\right)v=-ih,\,\,t\in\left[0,2\tilde{T}\right]\label{eq:3.10-1}
\end{equation}
Since $h$ is smooth and $\tilde{A}-F$ depends smoothly on $t$,
by the transposed Leibniz formula we conclude that $v\in C^{\infty}\left(\left[-\tilde{T},\tilde{T}\right];\mathscr{D}^{\prime}\left(\Omega^{\prime};\mathbb{C}^{r}\right)\right)$
(cf. \cite[Remark 1.2]{key-2}).\textcolor{red}{{} }By elliptic regularity
for the Laplacian $\triangle^{E}$, $u$ (and therefore also $v$)
is smooth in the interior of $N$, and so $\left.v\right|_{x^{n}=T}$
is smooth. Now, if we were to show that the solution to \eqref{eq:3.10-1}
is smooth for $t\in\left[0,2\tilde{T}\right)$ then we are done. Indeed,
we would have
\[
Id\cdot D_{x^{n}}u+i\tilde{A}\left(x,D_{x^{\prime}}\right)u=v\in C^{\infty}\left(\left(-\tilde{T},\tilde{T}\right],\times\Omega^{\prime};\mathbb{C}^{r}\right),
\]
and in particular, the restriction to the boundary$\left.v\right|_{x^{n}=0}$
is smooth. Now, if we set $R\sigma=\left.v\right|_{\partial N}$,
then
\[
\left.Id\cdot D_{x^{n}}u\right|_{\partial N}=-i\left.\tilde{A}u\right|_{\partial N}+R\sigma=-i\left.\tilde{A}\right|_{\partial N}\sigma+R\sigma=-i\left.A\right|_{\partial N}\sigma+R\sigma,
\]
and we will get the desired result since $R$ is a smoothing operator.
So in order to conclude the proof it is left to prove that $v$ is
smooth for $t\in\left[0,2\tilde{T}\right)$. We will do this in the
subsequent lemma.
\end{proof}
\begin{lem*}
There is an operator $B$ in the congruence class of $\tilde{A}-F$
which satisfies the conditions for a well-posed heat equation in Section
\ref{subsec:Well-posedness-of-the}. As a result, the solution $v$
to the equation \eqref{eq:3.10-1} is smooth for $t\in\left[0,2\tilde{T}\right)$.
\end{lem*}
\begin{proof}
We will start by checking the conditions for the operator$\tilde{A}-F$.
If some of them will not be satisfied then we will adjust the symbol
of $\tilde{A}-F$ to obtain $B$. The first condition is satisfied
due to the construction of $\tilde{A}$. Denote by $a_{1}=-Id\cdot\sqrt{q_{2}}$
and $a_{\leq0}$ the principal part and the reminder part, respectively,
of the full symbol of $\tilde{A}-F$ ($F$ has order zero). Let $\left\Vert \right\Vert $
be the (operator) norm on complex $r\times r$-matrices induced from
the Hermitian norm on $\mathbb{C}^{r}$. Note that for any matrix
$M$ we have $\left\Vert M\right\Vert \geq\left|\lambda\right|$,
where $\lambda$ is any eigenvalue of $M$, which implies that the
matrix 
\[
Id\cdot z-M
\]
is non-degenerate when $\left|z\right|>\left\Vert M\right\Vert $.
Indeed, its eigenvalues are equal to $z-\lambda$ and we have 
\[
\left|z-\lambda\right|\geq\left|z\right|-\left|\lambda\right|\geq\left|z\right|-\left\Vert M\right\Vert >0.
\]
Since $\tilde{A}-F$ is an elliptic PDO of order $1$ in $\Omega$
and $\Omega^{\prime}$ is precompact we have the following uniform
in $\left[-\tilde{T},\tilde{T}\right]\times\Omega^{\prime}$ bounds
\begin{gather}
c\left|\xi\right|\leq\left|a_{1}\right|\leq C_{1}\left(1+\left|\xi\right|^{2}\right)^{\nicefrac{1}{2}};\label{eq:3.10}\\
\left|a_{\leq0}\right|\leq C_{0},\label{eq:3.11}
\end{gather}
where $c$, $C_{0}$, and $C_{1}$ are some positive constants. Using
this we see that the matrix 
\begin{equation}
Id\cdot z-\frac{a_{1}+a_{\leq0}}{\left(1+\left|\xi\right|^{2}\right)^{\nicefrac{1}{2}}}\label{eq:3.12}
\end{equation}
is non-degenerate for $\left|z\right|>C_{1}+C_{0}$. Indeed, from
\eqref{eq:3.10},\eqref{eq:3.11} the norm of the quotient term is
bounded from above by the constant $C_{1}+C_{0}$. It is left to check
if \eqref{eq:3.12} is non-degenerate when $z=x+iy$ with $-\epsilon<x$,
for some sufficiently small $\epsilon>0$. We have
\begin{equation}
Id\cdot\left(x+iy+\frac{\sqrt{q_{2}}}{\left(1+\left|\xi\right|^{2}\right)^{\nicefrac{1}{2}}}\right)-\frac{a_{\leq0}}{\left(1+\left|\xi\right|^{2}\right)^{\nicefrac{1}{2}}}=Id\cdot\rho-A_{\leq0},\label{eq:3.14}
\end{equation}
where
\[
\rho=x+iy+\frac{\sqrt{q_{2}}}{\left(1+\left|\xi\right|^{2}\right)^{\nicefrac{1}{2}}},
\]
and
\[
A_{\leq0}=\frac{a_{\leq0}}{\left(1+\left|\xi\right|^{2}\right)^{\nicefrac{1}{2}}}.
\]
We know that the matrix \eqref{eq:3.14} is non-degenerate when $\left|\rho\right|>\left\Vert A_{\leq0}\right\Vert $.
Since $q_{2}$ can be arbitrarily small for $\xi$ close to $0$ we
cannot guarantee that $\rho$ will not vanish for any small $\epsilon>0$.
Therefore, we have to adjust the symbol of $\tilde{A}-F$. From \eqref{eq:3.10}
we obtain 
\[
\frac{\sqrt{q_{2}}}{\left(1+\left|\xi\right|^{2}\right)^{\nicefrac{1}{2}}}\geq\frac{c}{2},
\]
when $\left|\xi\right|\geq1$. Hence, when $\epsilon<\frac{c}{2}$
we have 
\[
\left|\rho\right|^{2}=y^{2}+\left(x+\frac{\sqrt{q_{2}}}{\left(1+\left|\xi\right|^{2}\right)^{\nicefrac{1}{2}}}\right)^{2}>\left(\frac{c}{2}-\epsilon\right)^{2},
\]
for $\left|\xi\right|\geq1$. From \eqref{eq:3.11} we see that 
\[
\left\Vert A_{\leq0}\right\Vert ^{2}\leq\frac{C_{0}^{2}}{1+\left|\xi\right|^{2}},
\]
which is less than $\left(\frac{c}{2}-\epsilon\right)^{2}$ when $\left|\xi\right|\geq C_{0}\left(\frac{c}{2}-\epsilon\right)^{-1}$.
Let $R=\max\left(1,C_{0}\left(\frac{c}{2}-\epsilon\right)^{-1}\right)$,
then the matrix \eqref{eq:3.14} is non-degenerate for $\left|\xi\right|\geq R$.
On the other hand, for $\left|\xi\right|<R$ we know that $\left\Vert A_{\leq0}\right\Vert \leq C_{0}$.
Now let us consider the congruent operator $B\left(t,x^{\prime},D_{x^{\prime}}\right)$
by adjusting the full symbol as
\[
a_{1}+a_{\leq0}-Id\cdot\psi\left(\xi\right),
\]
where $\psi\left(\xi\right)=Ce^{-\frac{\left|\xi\right|^{2}}{R^{2}}}$
and the constant $C$ is equal to $e\left(1+R^{2}\right)^{\nicefrac{1}{2}}\left(C_{0}+\frac{c}{2}\right)$.
One sees that now for $\left|\xi\right|<R$ we have 
\[
\left|\rho\right|^{2}=y^{2}+\left(x+\frac{\sqrt{q_{2}}+\psi\left(\xi\right)}{\left(1+\left|\xi\right|^{2}\right)^{\nicefrac{1}{2}}}\right)^{2}>\left(C_{0}+\frac{c}{2}-\epsilon\right)^{2}>C_{0}^{2}\geq\left\Vert A_{\leq0}\right\Vert ,
\]
so the matrix \eqref{eq:3.12} for the adjusted operator is non-degenerate
in this case. In the other cases it is clear that it remains non-degenerate,
which guarantees that the operator $B\left(t,x^{\prime},D_{x^{\prime}}\right)$
satisfies the conditions for a well-posed heat equation. The second
part of the lemma follows immediately. Indeed, since $B$ is congruent
to $\tilde{A}-F$ we have 
\[
Id\cdot\partial_{t}v-B\left(t\right)v\in C^{\infty}\left(\left[-\tilde{T},\tilde{T}\right]\times\Omega^{\prime};\mathbb{C}^{r}\right)
\]
 and, therefore, by Theorem \eqref{thm:Regularity=00005BTreves=00005D}
the solution $v$ is smooth for $t\in\left[0,2\tilde{T}\right)$. 
\end{proof}
\begin{cor}
\label{cor:Symbol of DtN} The full symbol of the DtN operator $\Lambda_{g,\nabla^{E}}$
is the same as the full symbol of $\left.A\left(x,D_{x^{\prime}}\right)\right|_{\partial N}$.
In particular, the DtN operator is a classical elliptic pseudodifferential
operator of order $1$.
\end{cor}

We are now in a position to recover the geometric data (the metric
$g$ and the connection $\nabla^{E}$) on the boundary $\partial N$
from the given DtN operator $\Lambda$.

\subsection{Proof of Theorem \ref{thm:1}}

Let $\left\{ x^{1},\dots,x^{n}\right\} $ denote boundary normal coordinates
associated with $\left\{ x^{1},\dots,x^{n-1}\right\} $ and $\left(\varepsilon_{1},\dots,\varepsilon_{r}\right)$
denote boundary normal frame defined by $\left(\epsilon_{1},\dots,\epsilon_{r}\right)$.
Note that since 

\[
\partial_{x^{n}}g_{kl}=-\underset{\eta,\mu}{\sum}g_{k\eta}\left(\partial_{x^{n}}g^{\eta\mu}\right)g_{\mu l},
\]
it also suffices to determine the inverse matrix $\left(g^{kl}\right)$
and its normal derivatives instead of $\left(g_{kl}\right)$ and its
normal derivatives. Using Corollary \prettyref{cor:Symbol of DtN}
and \eqref{eq:3.4} we get
\[
\lambda_{1}=-\sqrt{q_{2}},
\]
and we have at any $p\in\partial N$,
\[
q_{2}\left(p,\xi^{\prime}\right)=\stackrel[k,l=1]{n-1}{\sum}g^{kl}\left(p\right)\xi_{k}\xi_{l}=\left(\frac{\textrm{Trace}\,\lambda_{1}}{r}\right)^{2}.
\]
Thus, the principal symbol of the DtN operator determines $g^{kl}$
at each boundary point $p$.

Next from Corollary \prettyref{cor:Symbol of DtN} and \eqref{eq:3.5}
we have 
\begin{multline}
\lambda_{0}=\frac{1}{2\sqrt{q_{2}}}\left[\stackrel[l]{n-1}{\sum}\partial_{\xi^{l}}\sqrt{q_{2}}D_{x^{l}}\sqrt{q_{2}}-q_{1}-\partial_{x^{n}}\sqrt{q_{2}}+F\sqrt{q_{2}}\right]=\\
=-\frac{1}{2\sqrt{q_{2}}}\partial_{x^{n}}\sqrt{q_{2}}+\frac{i}{2\sqrt{q_{2}}}\stackrel[l=1]{n-1}{\sum}V^{l}\xi_{l}-\frac{1}{4}\stackrel[k,l=1]{n-1}{\sum}g^{kl}\left(x\right)\partial_{x^{n}}g_{kl}\left(x\right)+T_{0},\label{eq:3.7}
\end{multline}
where 
\[
T_{0}=\frac{1}{2\sqrt{q_{2}}}\stackrel[l]{n-1}{\sum}\partial_{\xi^{l}}\sqrt{q_{2}}D_{x^{l}}\sqrt{q_{2}}+\frac{i}{2\sqrt{q_{2}}}\stackrel[k,l=1]{n-1}{\sum}\left(\frac{1}{2}g^{kl}\left(x\right)\partial_{x^{k}}\log\delta\left(x\right)+\partial_{x^{k}}g^{kl}\left(x\right)\right)\xi_{l}
\]
is an expression involving only $g_{kl}$, $g^{kl}$, and their tangential
derivatives along $\partial N$. Note that $\sum g^{kl}g_{kl}=n-1$,
and so
\[
-\stackrel[k,l=1]{n-1}{\sum}g^{kl}\left(x\right)\partial_{x^{n}}g_{kl}\left(x\right)=\stackrel[k,l=1]{n-1}{\sum}g_{kl}\left(x\right)\partial_{x^{n}}g^{kl}\left(x\right).
\]
If we set $h^{kl}=\partial_{x^{n}}g^{kl}$, $h=\sum g_{kl}h^{kl}$,
and $\left\Vert \xi^{\prime}\right\Vert ^{2}=\sum g^{kl}\xi_{k}\xi_{l}=q_{2}$,
we can rewrite \eqref{eq:3.7} in the form
\begin{equation}
\lambda_{0}\left(\xi\right)=-\frac{1}{4\left\Vert \xi^{\prime}\right\Vert ^{2}}\stackrel[k,l=1]{n-1}{\sum}\left(h^{kl}-hg^{kl}\right)\xi_{k}\xi_{l}+\frac{i}{2\left\Vert \xi^{\prime}\right\Vert }\stackrel[l=1]{n-1}{\sum}V^{l}\xi_{l}+T_{0}.\label{eq:3.17}
\end{equation}
From antisymmetric part $\lambda_{0}\left(\xi\right)-\lambda_{0}\left(-\xi\right)$
we obtain 
\[
\frac{i}{\left\Vert \xi^{\prime}\right\Vert }\stackrel[l=1]{n-1}{\sum}V^{l}\xi_{l},
\]
which allows us to determine $V^{l}$ and multiplying by $\nicefrac{1}{2}g_{kl}$
we get
\[
\frac{1}{2}g_{kl}V^{l}=g_{kl}g^{lj}\omega_{j}=\delta_{k}^{j}\omega_{j}=\omega_{k},
\]
so we determined the connection matrix $\omega_{k}$ for $k=1,\dots,n-1$
on the boundary $\partial N$. Let us look at the remaining terms.
Only the first term in \eqref{eq:3.17} is unknown yet. But we can
recover it from all the other terms. Thus, we can recover the quadratic
form
\[
\kappa^{kl}=h^{kl}-hg^{kl}.
\]
When $n>2$, this in turn determines $h^{kl}=\partial_{x^{n}}g^{kl}$,
since 
\begin{equation}
h^{kl}=\kappa^{kl}+\frac{1}{2-n}\left(\underset{p,q}{\sum}\kappa^{pq}g_{pq}\right)g^{kl}.\label{eq:3.8}
\end{equation}

Now let us look at \eqref{eq:3.6}. We have
\begin{multline}
\lambda_{-1}=\frac{1}{2\sqrt{q_{2}}}\left[-\stackrel[l]{n-1}{\sum}\partial_{\xi^{l}}\sqrt{q_{2}}D_{x^{l}}\lambda_{0}-\stackrel[l]{n-1}{\sum}\partial_{\xi^{l}}\lambda_{0}D_{x^{l}}\sqrt{q_{2}}+\right.\\
\left.+\frac{1}{2}\stackrel[k,l]{n-1}{\sum}\partial_{\xi^{k}}\partial_{\xi^{l}}\sqrt{q_{2}}D_{x^{k}}D_{x^{l}}\sqrt{q_{2}}-q_{0}+\partial_{x^{n}}\lambda_{0}-F\lambda_{0}\right]\\
=-\frac{1}{2\sqrt{q_{2}}}q_{0}+\frac{1}{2\sqrt{q_{2}}}\partial_{x^{n}}\lambda_{0}+T_{-1}\left(g_{kl},\omega_{l}\right)=\\
-\frac{1}{8\left\Vert \xi^{\prime}\right\Vert ^{3}}\stackrel[k,l=1]{n-1}{\sum}\partial_{x^{n}}\kappa^{kl}\xi_{k}\xi_{l}+\frac{i}{4\left\Vert \xi^{\prime}\right\Vert ^{2}}\stackrel[l=1]{n-1}{\sum}\partial_{x^{n}}V^{l}\xi_{l}+T_{-1}\left(g_{kl},\omega_{l}\right),\label{eq:3.9}
\end{multline}
where $T_{-1}\left(g_{kl},\omega_{l}\right)$ is an expression involving
only $g_{kl},g^{kl}$, their first normal derivatives and the boundary
values of $\omega_{l}$. Here again looking at the antisymmetric part
of $\lambda_{-1}$ we determine $\partial_{x^{n}}V^{l}$ and subsequently
$\partial_{x^{n}}\omega_{k}$ for $k=1,\dots,n-1$. The first term
of \eqref{eq:3.9} is again determined by the other terms, which allows
us to recover
\[
-\frac{1}{8\left\Vert \xi^{\prime}\right\Vert ^{3}}\stackrel[k,l=1]{n-1}{\sum}\partial_{x^{n}}\kappa^{kl}\xi_{k}\xi_{l},
\]
and hence also $\partial_{x^{n}}\kappa^{kl}$. Due to \eqref{eq:3.8}
the latter determines $\partial_{x^{n}}h^{kl}=\partial_{x^{n}}^{2}g^{kl}$.

Proceeding by induction, let $m\le-1$, and suppose we have shown
that, when $-1\ge j\ge m$,
\[
\lambda_{j}=-\frac{1}{\left\Vert 2\xi^{\prime}\right\Vert ^{2-j}}\stackrel[k,l=1]{n-1}{\sum}\left(\partial_{x^{n}}^{\left|j\right|}\kappa^{kl}\right)\xi_{k}\xi_{l}+\frac{i}{\left\Vert 2\xi^{\prime}\right\Vert ^{1-j}}\stackrel[l=1]{n-1}{\sum}\partial_{x^{n}}^{\left|j\right|}V^{l}\xi_{l}+T_{j}\left(g_{kl},\omega_{l}\right),
\]
where $T_{j}\left(g_{kl},\omega_{l}\right)$ involves only the boundary
values of $g_{kl},g^{kl}$, their normal derivatives of order at most
$\left|j\right|$, and also for $j\le-1$ involves the boundary values
of $\omega_{l}$, and their normal derivatives of order at most $\left|j\right|-1$.
From Corollary \prettyref{cor:Symbol of DtN} and \eqref{eq:3.7-1}
we get 
\begin{multline*}
\lambda_{m-1}=\frac{1}{2\sqrt{q_{2}}}\left[\partial_{x^{n}}\lambda_{m}+\underset{\substack{j,k,K\\
m\le j,k\le1\\
\left|K\right|=j+k-m
}
}{\sum}\frac{1}{K!}\partial_{\xi}^{K}\lambda_{j}D_{x}^{K}\lambda_{k}-F\lambda_{m}\right]=\\
=-\frac{1}{\left\Vert 2\xi^{\prime}\right\Vert ^{3-m}}\stackrel[k,l=1]{n-1}{\sum}\left(\partial_{x^{n}}^{\left|m-1\right|}\kappa^{kl}\right)\xi_{k}\xi_{l}+\frac{i}{\left\Vert 2\xi^{\prime}\right\Vert ^{2-m}}\stackrel[l=1]{n-1}{\sum}\partial_{x^{n}}^{\left|m-1\right|}V^{l}\xi_{l}+T_{m-1}\left(g_{kl},\omega_{l}\right).
\end{multline*}
Taking the antisymmetric part of $\lambda_{m-1}$ we can determine
$\partial_{x^{n}}^{\left|m-1\right|}V^{l}$ and, therefore, also $\partial_{x^{n}}^{\left|m-1\right|}\omega_{k}$.
The first term is again determined by the other terms. So we can recover
$\partial_{x^{n}}^{\left|m-1\right|}\kappa^{kl}$ and thus $\partial_{x^{n}}^{\left|m-2\right|}g^{kl}$
also. This completes the induction step.

\subsection{Gauge equivalence of the reconstructed connection}

Let $\pi_{E}:E\rightarrow X$ and $\pi_{F}:F\rightarrow Y$ be two
smooth vector bundles over smooth manifolds. We say that an isomorphism
$\phi:E\rightarrow F$ covers a diffeomorphism $\psi:X\rightarrow Y$
if the relation
\[
\pi_{F}\circ\phi=\psi\circ\pi_{E}
\]
holds. Note that any isomorphism $\phi$ covers the unique underlying
diffeomorphism of the bases defined by $\psi=\pi_{F}\circ\phi\circ\pi_{E}^{-1}$.
One of the corollaries of the main result is the following Proposition
on gauge equivalence. 
\begin{prop}
Let $\left(N_{i},g_{i},E_{i},\nabla^{i}\right)$, where $i=1,2$,
be two Euclidean smooth vector bundles defined over connected compact
Riemannian manifolds with boundary. Suppose that for some open subsets
$\Sigma_{i}\subset\partial N_{i}$ there exists a vector bundle isomorphism
$\phi:\left.E_{1}\right|_{\Sigma_{1}}\rightarrow\left.E_{2}\right|_{\Sigma_{2}}$
that intertwines with the corresponding Dirichlet-to-Neumann operators
$\Lambda_{\Sigma_{1}}$ and $\Lambda_{\Sigma_{2}}$, i.e. the relation
$\phi\circ\Lambda_{\Sigma_{1}}\left(s\right)\circ\psi^{-1}=\Lambda_{\Sigma_{2}}\left(\phi\circ s\circ\psi^{-1}\right)$
holds for any smooth section $s$ of $\left.E_{1}\right|_{\Sigma_{1}}$.
Then the isomorphism $\phi$ is a gauge equivalence, $\phi^{*}\nabla^{2}=\nabla^{1}$,
and covers an isometry $\psi:\left(\Sigma_{1},g_{1}\right)\rightarrow\left(\Sigma_{2},g_{2}\right)$.
\end{prop}

\begin{proof}
Clearly, the isomorphism $\phi$ intertwining the DtN operators $\Lambda_{\Sigma_{1}}$
and $\Lambda_{\Sigma_{2}}$ is equivalent to having the equality of
operators 
\[
\Lambda_{\Sigma_{1}}\left(s\right)=\phi^{-1}\circ\Lambda_{\Sigma_{2}}\left(\phi\circ s\circ\psi^{-1}\right)\circ\psi.
\]
The operator on the right hand side is a natural pull-back of the
operator $\Lambda_{\Sigma_{2}}$ along $\phi$. Therefore, the metric
and connection on $\left(E_{1},\Sigma_{1}\right)$ reconstructed from
its full symbol are equal to $\psi^{*}g_{2}$ and $\phi^{*}\nabla^{2}$,
respectively. On the other hand the full symbols of the above two
operators coincide. Hence, we have $g_{1}=\psi^{*}g_{2}$ and $\nabla^{1}=\phi^{*}\nabla^{2}$,
which completes the proof.
\end{proof}

\subsection{The case of surfaces}

In two dimensions, we can only aim to reconstruct a conformal class
of metrics from the DtN operator. This is because the connection Laplacian
is conformally contravariant in two dimensions. Indeed, from the definition
of the connection Laplacian we have
\[
\triangle_{e^{\mu}g}^{E}=e^{-\mu}\triangle_{g}^{E},
\]
where the subscript indicates the metric used to define the connection
Laplacian. Clearly, $\triangle_{g}^{E}u=0$ if and only if $\triangle_{e^{\mu}g}^{E}u=0$.
The unit normal vector at the boundary for the conformal metric is
equal to 
\[
\left.e^{-\nicefrac{\mu}{2}}\right|_{\partial N}\cdot\frac{\partial}{\partial\nu}.
\]
Therefore, we have the following identity for the DtN operators 
\[
\Lambda_{e^{\mu}g,\nabla^{E}}=\left.e^{-\nicefrac{\mu}{2}}\right|_{\partial N}\cdot\Lambda_{g,\nabla^{E}}.
\]
This identity shows that the DtN operators constructed using conformal
metrics coincide if a conformal factor $e^{\mu}$ equals to $1$ at
the boundary. This fact poses obstacles to the recovery of the normal
derivatives of the geometric data at the boundary of surfaces. However,
we can still recover the metric and the connection on the boundary
from the symbol of the DtN operator. This follows immediately from
the proof of the Theorem \ref{thm:1}.

\end{document}